\documentclass{amsart}
\usepackage{amsfonts,amssymb,amscd,amsmath,enumerate,verbatim,newlfont,calc}
\usepackage{graphicx}

\newtheorem{theorem}{Theorem}[section]

\newtheorem{corollary}[theorem]{Corollary}
\newtheorem{example}[theorem]{Example}
\theoremstyle{definition}
\newtheorem{definition}[theorem]{Definition}
\newtheorem{remark}[theorem]{Remark}
\theoremstyle{notations}

\begin{document}
\title{Topological Homotopy Groups}
\author{H. Ghane}
\author{Z. Hamed}
\author{B. Mashayekhy}
\author{H. Mirebrahimi}
\address{Department of Mathematics, Ferdowsi University of Mashhad, P.O.Box 1159-91775, Mashhad, Iran}
 \maketitle
\begin{abstract}
D. K. Biss (Topology and its Applications 124 (2002) 355-371)
introduced the topological fundamental group and presented some
interesting basic properties of the notion. In this article we
intend to extend the above notion to homotopy groups and try to
prove some similar basic properties of the topological homotopy
groups. We also study more on the topology of the topological
homotopy groups in order to find necessary and sufficient conditions
for which the topology is discrete. Moreover, we show that studying
topological homotopy groups may be more useful than topological
fundamental groups.
\end{abstract}
\section{Introduction and Motivation}
Historically, J. Dugundji [3] in 1950, put, for the first time, a
topology on fundamental groups of certain spaces and deduced a
classification theorem for connected covers of a space.

Recently, Biss [1] generalized the results announced by J.
Dugundji. He equipped the fundamental group of a pointed space
$(X,x)$ with the quotient topology inherited from
$Hom((S^1,1),(X,x))$ with compact-open topology and denoted by
$\pi_1^{top}(X,x)$. He proved among other things that
$\pi_1^{top}(X,x)$ is a topological group which is independent of
the base point in path components and $\pi_1^{top}$ is a functor
from the homotopy category of based spaces to the category of topological
groups which preserves the direct product. He showed  that
$\pi_1^{top}$ is discrete if and only if the space $X$ is
semilocally simply connected. However, P. Fabel [5] mentioned that path
connectedness and locally path connectedness of $X$ is necessary.

P. Fabel [4], using Biss' results, showed that the topological
fundamental groups can distinguish the homotopy type of topological
spaces when the algebraic structures fail to do. On the other hand,
in some cases, this topology does not be able to separate spaces in
terms of their homotopy types. For examples, if $X$ is locally
contractible, then $\pi_1^{top}(X)$ is discrete which is not
interesting. In these situations the higher homotopy groups seems to
be useful and this is a motivation to define a natural topology on
homotopy groups of $X$. This topology, if it is compatible with the
topology of fundamental groups, can also distinguish the spaces with
distinct homotopy types even though the topology of their
topological fundamental groups are the same.

 In this article, we are going to extend some basic results of Biss
by introducing a topology on $n$th homotopy group of a pointed space
$(X,x)$ as a quotient space of $Hom((I^n,\dot{I}^n),(X,x))$ equipped
with compact-open topology and denote it by $\pi_n^{top}(X,x)$. We
will show that $\pi_n^{top}(X,x)$ is a topological group which does
not depend on based point $x$ in a path component. We will also
prove that $\pi_n^{top}$ is a functor from the homotopy category of
pointed spaces to the category of topological groups which preserves
the direct product. Moreover, we present the notion of an
$n$-semilocally simply connected space and prove that the
discreteness of $\pi_n^{top}(X,x)$ implies  that $X$ is
$n$-semilocally simply connected. By giving an example, it is shown
that the converse is not true, in general. However, we show that the
$n$th homotopy group of a locally $n$-connected mertrizable space
$X$ is discrete.

Fabel in [5], clarified the relationship between the cardinality of
$\pi_{1}(X,x)$ and discreteness of $\pi_{1}^{top}(X,x)$ and deduced
that if $X$ is a connected separable metric space and
$\pi_{1}^{top}(X,x)$ is discrete, then $\pi_{1}(X,x)$ is countable.
Here, we extend this result to higher homotopy groups and show that
the $n$th homotopy group of a connected, locally $n$-connected
separable metric space is countable. Finally, we give an example of
a metric space $X$ such that $\pi_1^{top}(X)$ is discrete whereas
$\pi_2^{top}(X)$ is not. This shows that studying topological
homotopy groups may be more useful than topological fundamental
groups.
\section{Topological Homotopy Groups}
Let $(X,x)$ be a pointed space. Then the space of continuous based
maps $Hom($ $(I^n,\dot{I}^n),(X,x))$ can be given the compact-open
topology; this topology has as a subbase the sets $\langle
K,U\rangle=\{f:(I^n,\dot{I}^n)\rightarrow(X,x)|f(K)\subseteq U\}$,
where $K$ ranges over all compact subsets of $I^n$ and $U$ ranges
over all open subsets of $X$. By considering the natural projection
\[Hom((I^n,\dot{I}^n),(X,x))\stackrel{p_n}{\twoheadrightarrow}[(I^n,\dot{I}^n),(X,x)]=\pi_n(X,x),\]
we are allowed to define a quotient topology on $\pi_n(X,x)$. By
$\pi_n^{top}(X,x)$ we mean the topological space $\pi_n(X,x)$
equipped with the above topology. In the following, we are going
to prove several basic properties of this topology.\\
\begin{theorem}
Let $(X,x)$ be a pointed space. Then $\pi_n^{top}(X,x)$ is a
topological group for all $n\geq 1$.
\end{theorem}
\begin{proof}
In order to show that the multiplication is continuous, we consider
the following commutative diagram
\[\begin{picture}(340,50)
\put(0,40){$Hom((I^n,\dot{I}^n),(X,x))\times
Hom((I^n,\dot{I}^n),(X,x))$}
\put(250,40){$Hom((I^n,\dot{I}^n),(X,x))$}
\put(50,0){$\pi_n^{top}(X,x)\times\pi_n^{top}(X,x)$}
\put(275,0){$\pi_n^{top}(X,x),$} \put(205,44){\vector(1,0){40}}
\put(155,4){\vector(1,0){115}} \put(100,36){\vector(0,-1){25}}
\put(300,36){\vector(0,-1){25}} \put(305,20){\small
$p_n$}\put(65,20){\small $p_n\times p_n$} \put(220,46){\small
$\tilde{m}_n$} \put(210,6){\small $m_n$}
\end{picture}\]
 where $\tilde{m}_n$ is concatenation of $n$-loops, and $m_n$ is
the multiplication in $\pi_n(X,x)$. Since $(p_n\times
p_n)^{-1}m_n^{-1}(U)=\tilde{m}_n^{-1}p_n^{-1}(U)$ for every open
subset $U$ of $\pi_n^{top}(X,x)$, it is enough to show that
$\tilde{m}_n$ is continuous. Let $\langle K,U\rangle$ be a basis
element in $Hom((I^n,\dot{I}^n),(X,x))$. Put
$$K_1=\{(t_1,\ldots,t_n)|(t_1,\ldots,t_{n-1},\frac{t_n}{2})\in K\}$$
and $$K_2=\{(t_1,\ldots,t_n)|(t_1,\ldots,t_{n-1},\frac{t_n+1}{2})\in
K\}.$$ Then
\[\tilde{m}_n^{-1}(\langle K,U\rangle)=\{(f_1,f_2)|(f_1*f_2)(K)\subseteq U\}=\langle K_1,U\rangle\times\langle K_2,U\rangle\]
is open in $Hom((I^n,\dot{I}^n),(X,x))\times
Hom((I^n,\dot{I}^n),(X,x))$ and so $\tilde{m}_n$ is continuous. To
prove that the operation of taking inverse is continuous, let $K$ be
any compact subset of $I^n$ and put
$$K^{-1}=\{(t_1,\ldots,t_{n-1},1-t_n)|(t_1,\ldots,t_n)\in K\}.$$
Clearly an $n$-loop $\alpha$ is in $\cap_{i=1}^m\langle
K_i,U_i\rangle$ if and only if its inverse is in
$\cap_{i=1}^m\langle K_i^{-1},U_i\rangle$, where $\langle
K_i,U_i\rangle$ are basis elements in $Hom((I^n,\dot{I}^n),(X,x))$.
Hence the inverse map is continuous.
\end{proof}
From now on, when we are dealing with $\pi_n^{top}$, by the notion
$\cong$ we mean the isomorphism in the sense of topological groups.
The following result shows that the topological group $\pi_n^{top}$
is independent of the base point $x$ in the path component.
\begin{theorem}
Let $\gamma:I\rightarrow X$ be a path with $\gamma(0)=x$ and
$\gamma(1)=y$. Then $\pi_n^{top}(X,x)\cong\pi_n^{top}(X,y)$.
\end{theorem}
\begin{proof}
Let $\alpha$ be a base $n$-loop at $x$ in $X$. Define the following
map
\[A_\alpha:I^n\times\{0\}\cup\dot{I}^n\times[0,1]\rightarrow X\]
by $A_\alpha(s,0)=\alpha(s)$ for all $s\in I^n$ and
$A_\alpha(s,t)=\gamma^{-1}(t)$ for all $(s,t)\in
\dot{I}^n\times[0,1]$. By gluing lemma $A_\alpha$ is a continuous
map. Now, we define
\[\gamma_\#:\pi_n^{top}(X,x)\longrightarrow\pi_n^{top}(X,y)\]
by $[\alpha]\longmapsto[A_\alpha\circ r(-,1)]$, where
$r:I^n\times[0,1]\rightarrow I^n\times[0,1]\cup\dot{I}^n\times[0,1]$
is the retraction introduced in [7]. It can be shown that
$\gamma_\#$ is a group isomorphism (see [7]). Now, it is enough to
show that $\gamma_\#$ is a homeomorphism. To prove $\gamma_\#$ is
continuous, consider the following commutative diagram
\[\begin{picture}(240,50)
\put(0,40){$Hom((I^n,\dot{I}^n),(X,x))$}
\put(150,40){$Hom((I^n,\dot{I}^n),(X,y))$}
\put(30,0){$\pi_n^{top}(X,x)$} \put(175,0){$\pi_n^{top}(X,y),$}
\put(100,44){\vector(1,0){45}} \put(80,4){\vector(1,0){90}}
\put(50,36){\vector(0,-1){25}} \put(200,36){\vector(0,-1){25}}
\put(205,20){\small $p_n$}\put(40,20){\small $p_n$}
\put(120,47){\small $\tilde{\gamma}$} \put(120,7){\small
$\gamma_\#$}
\end{picture}\]
 where $\tilde{\gamma}(\alpha)=A_\alpha\circ r(-,1)$ for all based
$n$-loop $\alpha$ at $x$ in $X$. For all basis elements $\langle
K,U\rangle$ of $Hom((I^n,\dot{I}^n),(X,x))$ we have
\begin{eqnarray*}
\tilde{\gamma}(\langle K,U\rangle)&=&\{\tilde{\gamma}|\
\alpha:(I^n,\dot{I}^n)\rightarrow(X,x),\ \alpha(K)\subseteq
U\}\\
&=&\{A_\alpha\circ r(-,1)|\ \alpha:(I^n,\dot{I}^n)\rightarrow(X,x),
\alpha(K)\subseteq U\}\\
&\subseteq&\{\beta|\ \beta:(I^n,\dot{I}^n)\rightarrow(X,y),\
\beta(K)\subseteq
U\}\\
&=&\langle K,U\rangle.
\end{eqnarray*}
 Hence $\tilde{\gamma}$ is continuous and so is $\gamma_\#$. It is
easy to see that $(\gamma^{-1})_\#$ is the inverse of $\gamma_\#$
and continuous. So the result holds.
\end{proof}
It is known that $\pi_n$ is a functor from the homotopy category of
based spaces to the category of groups. So it is natural to ask
whether $\pi_n^{top}$ is a functor. Suppose $f:(X,x)\rightarrow
(Y,y)$ is a pointed continuous map. It is enough to show that the
induced homomorphism
$f_*:\pi_n^{top}(X,x)\rightarrow\pi_n^{top}(Y,y)$ is continuous.
Consider the following commutative diagram
\[\begin{picture}(240,50)
\put(0,40){$Hom((I^n,\dot{I}^n),(X,x))$}
\put(150,40){$Hom((I^n,\dot{I}^n),(Y,y))$}
\put(30,0){$\pi_n^{top}(X,x)$} \put(175,0){$\pi_n^{top}(Y,y),$}
\put(100,44){\vector(1,0){45}} \put(80,4){\vector(1,0){90}}
\put(50,36){\vector(0,-1){25}} \put(200,36){\vector(0,-1){25}}
\put(205,20){\small $p_n$}\put(40,20){\small $p_n$}
\put(120,47){\small $f_\#$} \put(120,7){\small $f_*$}
\end{picture}\]
 where $f_\#(\alpha)=f\circ\alpha$. Clearly $f_\#^{-1}(\langle
K,U\rangle)=\langle K,f^{-1}(U)\rangle$ for all basis elements
$\langle K,U\rangle$ in $Hom((I^n,\dot{I}^n),(X,x))$. Since $f$ is
continuous, so are $f_\#$ and $f_*$. Hence we can consider
$\pi_n^{top}$ as a functor from $hTop^*$ to the category of
topological groups. Now, we intend to show that the functor
$\pi_n^{top}$ preserves the direct product of spaces.
\begin{theorem}
Let $\{(X_i,x_i)|i\in I\}$ be a family of  pointed spaces. Then
\[\pi_n^{top}(\prod_{i\in I}(X_i,x_i))\cong\prod_{i\in I}\pi_n^{top}(X_i,x_i).\]
\end{theorem}
\begin{proof}
Consider the commutative diagram
\[\begin{picture}(300,55)
\put(0,40){$M=\prod_{i\in I}Hom((I^n,\dot{I}^n),(X_i,x_i))$}
\put(200,40){$Hom((I^n,\dot{I}^n),(X_i,x_i))$}
\put(30,0){$N=\prod_{i\in I}\pi_n^{top}(X_i,x_i)$}
\put(210,0){$\pi_n^{top}\left(\prod_{i\in I}(X_i,x_i)\right),$}
\put(152,42){\vector(1,0){45}}\put(197,46){\vector(-1,0){45}}
\put(130,2){\vector(1,0){75}}\put(205,6){\vector(-1,0){75}}
\put(70,36){\vector(0,-1){25}} \put(250,36){\vector(0,-1){25}}
\put(255,20){\small $p_n$}\put(20,20){\small $Q=\Pi_{i\in I}p_n^i$}
\put(170,50){\small $\psi$} \put(170,10){\small $\varphi$}
\end{picture}\]
 where $\psi$ and $\varphi$ are natural homomorphism and
isomorphism, respectively (see [8]). To show that $\varphi$ is a
homeomorphism, it is enough to show that $Q$ is a quotient map. Let
$Q^{-1}(U)$ be an open subset of $M$ for some subset $U$ of $N$. By
product topology of $M$, $Q^{-1}(U)$ is the union of basic opens of
the form $\prod_{i\in I}V_i$, where $V_i$'s are open in
$Hom((I^n,\dot{I}^n),(X_i,x_i))$ and
$V_i=Hom((I^n,\dot{I}^n),(X_i,x_i))$ for all $i\notin J$, for some
finite subset $J$ of $I$. Suppose $V=\prod_{i\in J}V_i'$ is a
maximal open set having the property that $W=V\times\prod_{i\notin
J}Hom((I^n,\dot{I}^n),(X_i,x_i))$ is a subset of $Q^{-1}(U)$. Since
$W$ is open and maximal, it is of the form $Q^{-1}\left(\prod_{i\in
I}U_i\right)$, where $\prod_{i\in I}U_i$ is a basic open in $N$.
Since $Q^{-1}(U)$ is constructed with open sets like the above, we
have covered $U$ with open sets and so $U$ is open.
\end{proof}

\section{The topology of $\pi_n^{top}(X)$}
In this section, we are going to study more on the topology of
$\pi_n^{top}(X)$, specially we intend to find necessary and
sufficient conditions for which $\pi_n^{top}(X)$ is discrete.\\
\begin{definition}
A topological space $X$ is called $n$-semilocally simply connected
if  for each $x\in X$ there exists an open neighborhood $U$ of $x$
for which any $n$-loop in $U$ is nullhomotopic in $X$. In other
words the induced homomorphism of the inclusion
$i_*:\pi_n(U,x)\rightarrow\pi_n(X,x)$ is zero.
\end{definition}
\begin{theorem}
If $\pi_n^{top}(X)$ is discrete, then $X$ is $n$-semilocally simply
connected.
\end{theorem}
\begin{proof}
For each $x\in X$, since $\pi_n^{top}(X,x)$ is discrete, there
exists an open neighborhood $W$ in $Hom((I^n,\dot{I}^n),(X,x))$ of
the constant $n$-loop at $x$ such that each element of $W$ is
homotopic to the constant loop at $x$. By compact-open topology, we
can consider $W$ as $\cap_{i=1}^m\langle K_i,U_i\rangle$, where
$K_i$'s are compact subsets of $I^n$ and $U_i$'s are open in $X$.
Consider $U=\cap_{i=1}^m U_i$ as a nonempty open neighborhood, then
$\langle I^n,U\rangle\subseteq W$. Therefore any $n$-loop in $U$ at
$x$ belongs to $W$ and so is nullhomotopic in $X$. Hence $X$ is
$n$-semilocally simply connected.
\end{proof}
Note that the following examples show the inverse of Theorem 3.2 is
not true, in general. In both of them, we use the fact that the
compact-open topology on $Hom((I^2,\dot{I}^2),(X,0))$ is equivalent
to the uniformly convergence topology when $X$ is a metric space
[6].
\begin{example}
Let $X=\cup_{n\in N}S_{n}$, where
\[\ S_1=\{(x,y,z)\mid(x-\frac{1}{2})^{2}+y^{2}+z^{2}=\frac{1}{4}\},\]
\[\ S_n=\{(x,y,z)\mid(x-\frac{n-1}{2n})^{2}+y^{2}+z^{2}=(\frac{n-1}{2n})^{2}\},\]
\ \\
for each $n\geq2$. Then $\{S_{n}\}$ as a sequence of 2-loops in $X$
at $p=(0,0,0)$ uniformly converges to $S_1$. Now [$S_1$] is a limit
point in $\pi_2^{top}(X,x)$, nevertheless $X$ is 2-semilocally
simply connected.
\end{example}
\begin{example}
Let $X$ denotes the following subspace of $\mathbb{R}^3$:
\[X=[0,1]\times[0,1]\times\{0,1\}\cup[0,1]\times\{0,1\}\times[0,1]\cup\{0,1,\frac{1}{2},\frac{1}{3},\ldots\}\times[0,1]\times[0,1]\]
Let $p=(0,0,0)$. Consider the following sequence of $2$-loops at $p$
\[X_n=[0,\frac{1}{n}]\times[0,1]\times\{0,1\}\cup[0,\frac{1}{n}]\times\{0,1\}\times[0,1]\cup\{\frac{1}{n}\}\times[0,1]\times[0,1].\]
Obviously, this sequence is uniformly convergent to the
nullhomotopic loop $X_0=\{0\}\times[0,1]\times[0,1]$. Thus
$\pi_2^{top}(X)$ is not discrete, however one can see that $X$ is 2-
semilocally simply connected.
\end{example}
\newpage


\begin{figure}[t]
\centerline{\includegraphics[height=12cm,width=17cm]{cube.bmp}}
\vspace{-2cm} figure-1
\end{figure}


We recall the following definitions in [9].
\begin{definition}
A space $X$ is said to be $n$-connected for $n\geq0$ if it is path
connected and $\pi_{k}(X,x)$ is trivial for every base point $x\in
X$ and $1\leq k \leq n$. $X$ is called locally $n$-connected if for
each $x\in X$ and each neighborhood  $U$ of $x$, there is a
neighborhood $V \subseteq U \subseteq X$ containing $x$ so that
$\pi_{k}(V)\longrightarrow \pi_{k}(U)$ is zero map for all $0 \leq k
\leq n$ and for all basepoint in $V$.
\end{definition}
\begin{theorem}
Let $X$ be a locally $n$-connected metric space. Then for any $x\in
X$, $\pi_n^{top}(X,x)$ is discrete.
\end{theorem}
\begin{proof}
We know that $\pi_n^{top}(X,x)$ is the set of path components of
loop space $\Omega^{n}(X,x)$ topologized with the quotient topology
under the canonical surjection $p_{n}$ satisfying
$p_{n}(f)=p_{n}(g)$ if and only if the n-loops $f$ and $g$ belong to
the same path component of  $\Omega^{n}(X,x)$, see[7,Lemma 2.5.5].
To prove $\pi_n^{top}(X,x)$ is discrete,it is sufficient to show
that the path components of $\Omega^n(X,x)$ are open. Suppose $f\in
\Omega^n(X,x)$ and $f_k\rightarrow f$ uniformly. We must prove that
$f$ and $f_k$ are homotopic, for sufficiently large $k$.

 Since $X$ is locally $n$-connected and $Im(f)$ is compact, so
there exists $\epsilon > 0$ such that if $x\in Im(f)$ and $\alpha_x$
is an $m$-loop based at $x$ $(m\leq n)$, with
$diam(\alpha_x)<\epsilon$, then $\alpha_x$ is null-homotopic.

Since $f\cup\{f_1,f_2,\cdots\}$ is an equicontinuous collection of
maps and $Im(f)$ and $Im(f_k)$'s are compact, then for the
$\epsilon$ as above, there exists $\delta>0$ such that each subcube
$\widetilde{I}\subseteq I^n$ with $diam(\widetilde{I})<\delta$, the
images $f(\widetilde{I})$ and $f_k(\widetilde{I})$'s have diameters
less than $\epsilon$. Take a partition $\{I_1^n,\cdots,I_l^n\}$ of
$I^n$, with $diam(I_i^n)<\delta$, $i=1,\cdots,l$, then
$diam(f(I_i^n))<\epsilon$ and $diam(f_k(I_i^n))<\epsilon$.

Let $v_{j}^{i}, j=1,...,2^{n}$ be the vertices of $I_{i}^{n}$.
First, by local path connectivity of $X$, we can connect the
vertices $f(v_{j}^{i})$ and $f_{k}(v_{j}^{i})$ by small path, for
each $i=1,...,l$ and $j=1,...,2^{n}$ and sufficiently large $k$. The
boundary of the rectangles with corners  $f(v_{j}^{i})$ and
$f_{k}(v_{j}^{i})$ induce 1-loops which are homotopic to the
constant loop. By local $n$-connectivity of $X$, we can fill in the
homotopy across the sides of these rectangles. Similarly and by
induction on $k$, we construct inessentials $k$-loops for each
$k\leq n$ and then again fill in the homotopy across the sides of
induced $k$-rectangles. In this way, in the $n$th step, we obtain a
homotopy from $f\mid I_{i}^{n}$ to $f_{k}\mid I_{i}^{n}$. Now, the
gluing lemma yields a homotopy from $f$ to $f_{k}$.
\end{proof}
With the added assumption that $X$ is locally $(n-1)$-connected, the
inverse of the Theorem 3.2 holds.
\begin{theorem}
Suppose $X$ is a locally $(n-1)$-connected metrizable space and
$x\in X$. Then the following are equivalent:

(1) $\pi_{n}^{top}(X,x)$ is discrete.

(2) $X$ is $n$-semilocally simply connected at $x$.
\end{theorem}
The following result presents relationship between the cardinality
of $\pi_{n}(X,x)$ and discreteness of $\pi_{n}^{top}(X,x)$.
\begin{theorem}
Suppose $X$ is a connected separable metric space such that
$\pi_n^{top}(X,x)$ is discrete. Then $\pi_n(X,x)$ is countable.
\end{theorem}
\begin{proof}
Since $X$ is a separable metric space, it follows from the proof of
the Urysohn metrization theorem $[6, Theorem 4.1]$ that $X$ can be
embedded as a subspace of the Hilbert cube
$Q=\Pi_{i=1}^{\infty}[0,1]$. The space $Hom(I^n,Q)$ is separable and
metrizable, and hence the subspace $Hom((I^n,\dot{I}^n),(X,x))$ is
separable. Since $\pi_n(X,x)$ is the continuous image of
$Hom((I^n,\dot{I}^n),(X,x))$, the space $\pi^{top}_n(X,x)$ is
separable. In particular, if $\pi^{top}_n(X,x)$ is discrete, then
$\pi_n(X,x)$ is countable since $\pi^{top}_n(X,x)$ is the only dense
subspace of $\pi^{top}_n(X,x)$.
\end{proof}
\begin{corollary}
If $X$ is a connected, locally $n$-connected separable metric space,
then $\pi_n(X,x)$ is countable.
\end{corollary}
\begin{proof}
By Theorems 3.6 and 3.8, the result follows immediately.
\end{proof}
\begin{remark}
It is known that if $X$ is path connected, locally path connected,
semilocally simply connected and $p:\tilde{X}\rightarrow X$ is a
covering space of $X$, then
$p_*:\pi_1^{top}(\tilde{X})\rightarrow\pi_1^{top}(X)$ is an
embedding, see [1,4,5]. However, by our results, if $X$ is locally
$n$-connected
 then
$p_*:\pi_n^{top}(\tilde{X})\rightarrow\pi_n^{top}(X)$ is also an
embedding.
\end{remark}
\ \ \ Recall a well-known theorem in Algebraic Topology [8]
asserting the isomorphism
$$\pi_{n}(X,x)\cong\pi_{1}(\Omega^{n-1}(X,x),\widetilde{x}),$$ where $\widetilde{x}$ is the constant
loop $\widetilde{x}(t)=x$ and $\Omega^{n-1}(X,x)$ is the
$(n-1)$-loop space of $X$ at $x$ equipped with compact-open
topology. Using the above isomorphism we can consider another
topology on $\pi_{n}(X,x)$ induced by
$\pi_{1}^{top}(\Omega^{n-1}(X,x),\widetilde{x})$. We denote this
topological group by $ \pi_{n}^{\Omega}(X,x)$. Note that
$\pi_{n}^{\Omega}$ is the composition of the two functors
$\pi_{1}^{top}$ and $\Omega^{n-1}$ [1,6]. Therefore the operation
$\pi_{n}^{\Omega}$ is a functor from the homotopy category of
topological based spaces to the category of topological groups. Now,
it is natural and interesting to ask the relationship between the
two topological groups $\pi_{n}^{top}(X,x)$and
$\pi_{n}^{\Omega}(X,x)$. Consider the following diagram:
\[\begin{picture}(300,55)
\put(11,40){$Hom((I^n,\dot{I}^n),(X,x))$}
\put(180,40){$Hom((I,\dot{I}),(\Omega^{n-1}(X,x),\widetilde{x}))$}
\put(35,4){$\pi_{n}^{top}(X,x)$} \put(220,5){$\pi_n^{\Omega}(X,x),$}
\put(120,40){\vector(1,0){50}} \put(90,8){\vector(1,0){120}}
\put(50,36){\vector(0,-1){20}} \put(240,36){\vector(0,-1){20}}
\put(245,20){\small $q_2$} \put(55,20){\small $q_1$}
\put(145,45){\small $\psi$} \put(145,13){\small $\eta$}
\end{picture}\]
 where $q_1$ and $q_2$ are quotient maps which have been defined
and the mapping $\psi$ maps  $f:((I\times S^{n-1} )/\sim
,*)\longrightarrow (X,x)$ to
$f^{\#}:(I,\dot{I})\longrightarrow(\Omega^{n-1}(X,x),\widetilde{x})$,
for which $f^{\#}(t)=f^{\#}_{t}\in\Omega^{n-1}(X,x)$ and
$f^{\#}_{t}(z)=f([t,z])$ for each $z\in S^{n-1}$ ( note that
$(S^n,1)\thickapprox (I\times S^{n-1} )/\sim ,*)$ ). By [8, Theorem
11.12], there exists a bijection
$\eta:\pi_{n}^{top}(X,x)\longrightarrow\pi_n^{\Omega}(X,x)$ which
commutes the above diagram. Thus, in order to show that $\eta$ is a
homeomorphism, it is enough to show that $\psi$ is a homeomorphism.
Now suppose $X$ is a metric space, since $I$ is locally compact and
Hausdroff then the two compact-open topologies of the right hand
sight of the above diagram are equivalent to the uniform convergence
topology [6]. It is easy to see that a sequence $\{f_{n}\}$ is
convergent to $f$ in Hom$((I^{n},\dot{I}^{n}),(X,x))$ if and only if
the sequence $\{f_{n}^{\#}\}$ is convergent to $f^{\#}$ in
Hom$((I,\dot{I}),\Omega^{n-1}(X,x))$. Hence the two topological
groups $\pi_{n}^{top}(X,x)$ and $\pi_{n}^{\Omega}(X,x)$ are
isomorphic and so we can consider two topologies on $\pi_{n}(X,x)$
which are equivalent when $X$ is a metric space. By homeomorphism
$\pi_{n}^{top}(X,x)\cong\pi_{1}^{top}(\Omega^{n-1}(X,x))$, we have
the following assertions, (see [1,5]).

(a) Let $\Omega^{n-1}(X,x)$ be path connected, locally path
connected and semilocally simply connected. Then
$\pi_{n}^{top}(X,x)$ is discrete.

(b) Suppose $\Omega^{n-1}(X,x)$ is a connected separable space such
that $\pi_{n}^{top}(X,x)$  is discrete. Then $\pi_{n}(X,x)$ is
countable.

(c) Let $\Omega^{n-1}(X,x)$ be connected, locally path connected and
separable. Also, let $\pi_{n}(X,x)$ be free, then
$\pi_{n}^{top}(X,x)$ is discrete.

\ \\
 The following example shows that studying topological homotopy
groups may be more useful than topological fundamental groups.
\begin{example}\label{e-3}
Let $X=\cup_{n\in\mathbb{N}}S_n$, where
$S_n=\{(x,y,z)|(x-\frac{1}{n})^2+y^2+z^2=\frac{1}{n^2}\}$, be a
subspace of $\mathbb{R}^3$. It is easy to see that each $1$-loop in
$X$ is nullhomotopic (note that $S_n$'s are simply connected).
Therefore $\pi_1^{top}(X)$ is trivial.   the sequence $\{[S_n]\}$ is
convergent to identity element of $\pi_2^{top}(X,0)$, implying
that $\pi_2^{top}(X,0)$ is not discrete (see figure-2).\\
\end{example}
\begin{remark}
In general, if $\{X_i\}$ forms an inverse system of topological
spaces for which each $X_i$ contains an essential $n$-loop, then
$\pi_n^{top}\,(\displaystyle{\lim_\leftarrow} X_i) $ is not
discrete.
\end{remark}
\newpage
\begin{figure}[h]
    \centerline{\includegraphics[height=11cm,width=17cm]{sphere1.bmp}}
figure-2
        \end{figure}

\end{document}